\documentclass[12pt]{article}
\usepackage{latexsym,bm}
\usepackage{mathrsfs,amsmath,amssymb,amsthm}
\theoremstyle{plain}
\newtheorem{theorem}{Theorem}

\newtheorem{lemma}[theorem]{Lemma}

\theoremstyle{definition}

\theoremstyle{remark}
\newtheorem{remark}[theorem]{Remark}

\usepackage{amscd}

\newcommand{\la}{\langle}
\newcommand{\ra}{\rangle}

\makeatletter
\newcommand{\trace}{\mathop{\operator@font Trace}}
\newcommand{\vspan}{\mathop{\operator@font Span}}
\newcommand{\Int}{\mathop{\operator@font Int}}
\newcommand{\grad}{\mathop{\operator@font grad}}
\newcommand{\nor}{\mathop{\operator@font nor}}
\newcommand{\Sp}{\mathop{\operator@font Spec}}
\newcommand{\End}{\mathop{\operator@font End}}
\newcommand{\diver}{\mathop{\operator@font div}}
\newcommand{\Ric}{\mathop{\operator@font Ric}}

\begin{document}

\title{$\mathfrak A$-principal Hopf hypersurfaces in complex quadrics
}

\author{Tee-How \textsc{Loo}
\\
Institute of Mathematical Sciences, University of Malaya \\
50603 Kuala Lumpur, Malaysia.	\\ 
\ttfamily{looth@um.edu.my}
}

\date{}
\maketitle

\abstract{
A real hypersurface in the complex quadric  $Q^m=SO_{m+2}/SO_mSO_2$  is said to be $\mathfrak A$-principal if its unit normal vector field is 
singular of type $\mathfrak A$-principal everywhere.
In this paper, we show that a $\mathfrak A$-principal Hopf hypersurface in 
$Q^m$, $m\geq3$ is an open part of a tube around a totally geodesic $Q^{m+1}$ in $Q^m$.
We also show that such real hypersurfaces are the only contact real hypersurfaces in $Q^m$.
The classification for  pseudo-Einstein real hypersurfaces in $Q^m$, $m\geq3$, is also obtained.



\section{Introduction.}
A natural research problem that arise in the theory of Riemannian submanifolds, when the ambient spaces are equipped with some additional geometric structures, is to study the interactions between these structures and the submanifold structure on its submanifolds.

For real hypersurfaces in a Hermitian manifold with complex structure $J$,
a geometric condition naturally being considered is to require the line bundle $JT^\perp M$ over $M$ to be invariant under the shape operator $S$ of $M$, that is, $SJT^\perp M\subset JT^\perp M$.
Such real hypersurfaces are known as Hopf hypersurfaces and possess some interesting geometric properties, for instance, Hopf hypersurfaces in a complex projective space $\mathbb CP^m$ are curvature adapted and can be realized as tubes around complex submanifolds in $\mathbb CP^m$ (cf. \cite{borisenko}).   

Similar research has been carried out for real hypersurfaces in quarternionic Kaehler manifolds.
Martinez and Perez  classified real hypersurfaces $M$ with constant principal curvatures in quarternionic projective spaces $\mathbb HP^m$ of which the vector bundle 
$\mathfrak JT^\perp M$ over $M$ is invariant under the shape operator $S$ of $M$, where $\mathfrak J$ is the quarternionic Kaehler structure of $\mathbb HP^m$ (cf. \cite{martinez-perez}).
This results has been improved in \cite{berndt} by removing the constancy assumption of the principal curvatures.

The complex two-plane Grassmannian $G_2(\mathbb C^{m+2})$ is the uniques compact Kaehler, quarternionic Kaehler manifold with positive scalar curvature. Two natural conditions to be considered are both $JT^\perp M$ and $\mathfrak JT^\perp M$ are invariant under the shape operator $S$ of real hypersurfaces.
Berndt and Suh used these properties to characterized tubes around $G_2(\mathbb C^{m+1})$ and tubes around $\mathbb HP^{m/2}$ in $G_2(\mathbb C^{m+2})$ (cf. \cite{berndt-suh1}).
An extension 
to the non-compact dual of $G_2(\mathbb C^{m+2})$
can be found in \cite{berndt-suh}. 

In this paper, we study real hypersurfaces, in the $m$-dimensional complex quadric $Q^m=SO_{m+2}/SO_mSO_2$, 
$m\geq2$. The complex quadric $Q^m$, is a Hermitian symmetric space of rank two.
It is the only compact non-totally geodesic parallel complex hypersurface in the complex projective space $\mathbb CP^{m+1}$ (cf. \cite{nakagawa-takagi}).
This property determines on $Q^m$, on top of the complex structure $J$, 
another distinguished geometric structure $\mathfrak A$ on $Q^m$, which is.
 a complex line bundle over $Q^m$ generated by conjugations on tangent spaces of $Q^m$ induced by the shape operator of $Q^m$ in $\mathbb CP^{m+1}$.

With respect to the structure $\mathfrak A$, there are two types of singular tangent vectors for $Q^m$, namely, 
$\mathfrak A$-principal and $\mathfrak A$-isotropic singular tangent vectors.
A real hypersurface $M$ in $Q^m$ is said to be \emph{$\mathfrak A$-principal} (resp. \emph{$\mathfrak A$-isotropic})
if the normal bundle of $M$ consists of $\mathfrak A$-principal (resp. $\mathfrak A$-isotropic) singular tangent vectors in $Q^m$.

Typical examples for $\mathfrak A$-principal (resp. $\mathfrak A$-isotropic) real hypersurfaces are the tubes around 
totally geodesic $Q^{m-1}$ (resp. $\mathbb CP^k$, $m=2k$ is even) in $Q^m$.
These real hypersurfaces have a number of interesting geometric properties,
for instance, both of them are Hopf and tubes around totally geodesic $\mathbb CP^{k}$ in $Q^{2k}$ are the only real hypersurfaces in $Q^{2k}$ with isometric Reeb flow (cf. \cite{berndt-suh2}), while tubes around $Q^{m-1}$ in $Q^m$ appeared to be the only known examples of contact real hypersurfaces in $Q^m$ (cf. \cite{berndt1}).

This raises two interesting problems:  classifying  $\mathfrak A$-principal Hopf hypersurfaces and $\mathfrak A$-isotropic Hopf hypersurfaces in the complex quadric $Q^m$. In this paper, we first study the formal problem and 
show that these real hypersurfaces are indeed tubes around $Q^{m-1}$ in $Q^m$
(cf. Theorem~\ref{thm:main}).

Let $M$ be a real hypersurface in a K\"ahler manifold $\hat M$. 
Denote by $\Phi(\cdot,\cdot):=\la\cdot,\phi~\cdot\ra$ the fundamental $2$-form.
If there exist a non-zero function $\rho$ on $M$ such that $d\eta=\rho\Phi$, then $M$  admits a contact structure.
In this case, we called $M$ a \emph{contact real hypersurface} in $\hat M$.
In  \cite{berndt1},  Berndt asked whether  tubes around totally geodesic $Q^{m-1}$ are the only contact real hypersurfaces in $Q^m$.
We shall give an affirmative answer for this question (cf. Theorem~\ref{thm:contact}).

A almost contact metric manifold $M$ is said to be \emph{pseudo-Einstein} if there exist constants $a$, $b$ such that 
its Ricci tensor $\Ric$ is given by 
\begin{align*} 
\Ric X=&aX+b\eta(X)\xi.	
\end{align*}
Pseudo-Einstein real hypersurfaces in non-flat complex space forms were studied in \cite{cecil-ryan, ivey-ryan, kim-ryan, kon, montiel}.
We study pseudo-Einstein real hypersurfaces in $Q^m$, $m\geq 3$, and 
show that a complete pseudo-Einstein real hypersurfaces must be a special kind of  tubes around totally geodesic $Q^{m-1}$ 
(cf. Theorem~\ref{thm:Einstein-2}).

This paper is organized as follows:
In Sect. 2,
we take a quick revision on the geometric structures on complex quadrics $Q^m$.
In Sect. 3, we fix notations and establish a general framework for understanding geometry of real hypersurfaces $M$ in $Q^m$.
We derive some general identities for Hopf hypersurfaces in $Q^m$ in Sect. 4.
In particular, we show that the only Hopf hypersurfaces with constant Reeb principal curvature are the  $\mathfrak A$-principal and $\mathfrak A$-isotropic ones (cf. Lemma~\ref{lem:7}).
The main results are proved in the last three sections. 
In Sect. 5--6, we show that
the following statements are equivalent.
\begin{enumerate}
\item $M$ is an open part of a tube around a totally geodesic $Q^{m-1}$ in $Q^m$.
\item $M$ is a $\mathfrak A$-principal Hopf hypersurface in $Q^m$.
\item $M$ is contact real hypersurface in $Q^m$.
\end{enumerate}
We study pseudo-Einstein real hypersurfaces  in Sect. 7. 
A classification for a complete pseudo-Einstein real hypersurfaces in $Q^m$ is obtained.


\section{The complex quadrics}

We denote by $\mathbb CP^{m+1}$ the $(m+1)$-dimensional complex projective space of constant holomorphic sectional curvature $4$ with respect to the Fubini-Study metric $\la,\ra$.
Each point $[z]\in\mathbb CP^{m+1}$ can be regarded as a complex line in $\mathbb C^{m+2}$ spanned by 
$z\in\mathbb C_{\times}^{m+1}$. Up to identification, the tangent space $T_{[z]}\mathbb CP^{m+1}$ is given by
\[
T_{[z]}\mathbb CP^{m+1}=\mathbb C^{m+2}\ominus[z]=\{w\in\mathbb C^{m+2}:\la w,z\ra_{\mathbb C}=0\}
\]
where $\la,\ra_{\mathbb C}$ is the Hermitian inner product on $\mathbb C^{m+2}$.

The $m$-dimensional complex quadric $Q^m$ is a complex hypersurface defined by the quadratic equation $z_0^2+z_1^2+\cdots+z_{m+1}^2=0$ in $\mathbb CP^{m+1}$, which is isometric to the real Grassmannian of oriented two-planes of $\mathbb R^{m+2}$ and is a compact Hermitian symmetric space of rank two.

We denote by $J$ both the complex structure of $\mathbb CP^{m+1}$ and that induced on $Q^m$, and by  
$\la ,\ra$ as well the induced metric tensor on $Q^m$. 
As $Q^2$ is isometric to $S^2\times S^2$, we will consider $m\geq3$ in the main part of the paper.

At each $[z]\in Q^m$, up to identification, the normal space $T_{[z]}^\perp Q^m=[\bar z]$ and 
tangent space
$T_{[z]}Q^m=\mathbb C^{m+2}\ominus([z]\oplus[\bar z])$. 
Denote by $A_\zeta$ the shape operator of $Q^m$ in $\mathbb CP^{m+1}$ with respect to 
a unit vector $\zeta \in T_{[z]}^\perp Q^m$. It is known that 
$A_\zeta$ is a self-adjoint involution on $T_{[z]}Q^m$ and satisfies $A_\zeta J+JA_\zeta=0$. 
In other word, $A_\zeta$ is a conjugation on $T_{[z]}Q^m$ with respect to the Hermitian metric $\la ,\ra_{\mathbb C}$
given by
\[
\la X,Y\ra_{\mathbb C}=\la X,Y\ra+\sqrt{-1}\la X,JY\ra
\]
for any $X,Y\in T_{[z]}Q^m$.

Let $V(A_\zeta)$ (resp.  $JV(A_\zeta)$) be the $(+1)$-eigenspace (resp. the $(-1)$-eigenspace) of $A_\zeta$. Then we have
\[
T_{[z]}Q^m=V(A_\zeta)\oplus JV(A_\zeta),
\]
and $A_\zeta$ defines a real structure $V(A_\zeta)$ on $T_{[z]}Q^m$.
In particular, for $\zeta=\bar z$, the shape operator $A_{\bar z}w=-\bar w$,
for each $w\in T_{[z]}Q^m$, 
$V(A_{\bar z}) =\mathbb R^{m+2}\cap  T_{[z]}Q^m$ and 
$JV(A_{\bar z}) =\sqrt{-1}\mathbb R^{m+2}\cap  T_{[z]}Q^m$
(cf. \cite{reckziegel}).

The $\mathbb C Q$-structure $\mathfrak A^0_{[z]}:=\{\lambda A_\zeta : \lambda \in S^1\}$  
on $T_{[z]}Q^m$ is independent of the choice of $\zeta$ as every unit vectors $\zeta, \zeta'\in T_{[z]}^\perp Q^m$ can be related by
$\zeta'=\lambda\zeta$ for some $\lambda\in S^1$.
It follows that $\mathfrak A^0=\cup_{[z]\in Q^m}\mathfrak A^0_{[z]}$ is a $S^1$-bundle over $Q^m$.
We can also construct a complex line bundle $\mathfrak A=\cup_{[z]\in Q^m}\mathbb C\otimes A_\zeta$.
over $Q^m$.

We corresponds to each unit vector field $\zeta$ normal to $Q^m$ in $\mathbb CP^{m+1}$ a section $A_\zeta$ of 
$\mathfrak A$.
Denote by $\hat\nabla$ and $\nabla^\perp$ the connections corresponding to $TQ^m$ and $T^\perp Q^m$ respectively, induced by the Levi-Civita connection of $\mathbb CP^{m+1}$. 
For vectors $X,Y$ tangent to $Q^m$ in $\mathbb CP^{m+1}$, we have
$\nabla^\perp_X\zeta=q_\zeta(X)J\zeta$, for some $1$-form $q_\zeta$ on $Q^m$.
Since $Q^m$ is a parallel complex hypersurface in $\mathbb CP^{m+1}$ and $A_{J\zeta}=JA_\zeta$, 
we have
\[
0=(\hat\nabla_X A)_\zeta Y=\hat\nabla_XA_\zeta Y-A_\zeta\hat\nabla_XY-A_{\nabla^\perp_X\zeta}Y 
=(\hat\nabla_XA_\zeta)Y-q_\zeta(X)JA_\zeta Y.
\] 
It follows that for each section $A$  of $\mathfrak A$, there exists a $1$-form $q$ on 
$Q^m$ such that 
\begin{align}\label{eqn:delA}
\hat\nabla A=JA\otimes q.
\end{align}
This implies that the subbundle $\mathfrak A$ of $\End(TQ^m)$ is parallel.

A non-zero vector $W\in T_{[z]}Q^m$ is said to be \emph{singular} if it is tangent to more than one maximal flat in $Q_m$. 
There are two types of singular tangent vectors for the complex quadric $Q^m$:
$\mathfrak A$-principal singular and $\mathfrak A$-isotropic singular.
A singular tangent vector $W$ is said to be \emph{$\mathfrak A$-principal} if there exists a conjugation $A\in\mathfrak A_{[z]}$ such that $W\in V(A)$.
If $\la AW,W\ra=\la AW,JW\ra=0$ for some (and then, for all) $A\in\mathfrak A_{[z]}$ then 
$W$ is called a \emph{$\mathfrak A$-isotropic singular vector}.

We have the following characterizations for $\mathfrak A$-principal singular tangent vectors.
\begin{lemma}
Let $W\in T_{[z]}Q^m$ be a unit vector. Then following are equivalent:
\begin{enumerate}
\item[(a)] $W$ is $\mathfrak A$-principal.
\item[(b)] There exists (and hence for all) $A\in\mathfrak A_{[z]}$ such that $AW\in\mathbb C W$. 
\item[(c)] For each $A\in\mathfrak A_{[z]}^0$, $\la AW,W\ra^2+\la AW,JW\ra^2=1$. 
\end{enumerate}
\end{lemma}

In general, for each unit tangent vector $W\in T_{[z]}Q^m$ and $A\in\mathfrak A_{[z]}^0$, 
we can write 
\[
W=\cos(t) X+\sin(t)JY
\]
where $X,Y\in V(A)$ are orthonormal vectors and $t\in[0,\pi/4]$.
$W$ is a $\mathfrak A$-principal (resp. $\mathfrak A$-isotropic) singular tangent vectors when $t=0$ (resp. $t=\pi/4$).  

From the Gauss equation of the complex hypersurface $Q^m$ in $\mathbb CP^{m+1}$, the curvature tensor $\hat R$ of
$Q^m$ is given by
\begin{align}\label{eqn:hatR}
\hat R(X,Y)=X\wedge Y+JX\wedge JY-2\la JX,Y\ra J+AX\wedge AY+JAX\wedge JAY
\end{align} 
for any $X,Y$ tangent to $Q^m$ and $A$ in $\mathfrak A^0$,
where $(U\wedge V)Z=\la V,Z\ra U-\la U,Z\ra V$. 


\section{Real  hypersurfaces in $Q^m$}

Let $M$ be a connected real hypersurface in $Q^m$, and let $N$ be a (local) unit vector field normal to $M$.
We define $\xi:=-JN$, $\eta$ the $1$-form dual to $\xi$ and $\phi:=J_{|TM}-\xi\otimes\eta$.
Then $(\phi,\xi,\eta)$ is an almost contact structure on $M$, that is,
\begin{equation*}\label{eqn:contact}
	\phi ^{2}X = -X + \eta (X) \xi,
	\quad \phi \xi = 0,
	\quad \eta (\phi X) = 0, 
	\quad \eta (\xi ) = 1. 
\end{equation*}
Denote by $\nabla$ the Levi-Civita connection, $\la,\ra$ the induced Riemannian metric and $S$ the shape operator of $M$. Then 
\begin{equation}\label{eqn:delxi}
	(\nabla_{X} \phi)Y=\eta(Y)SX-\la SX,Y\ra\xi,
	\quad \nabla_X \xi = \phi SX.
\end{equation}
for any $X,Y$ tangent to $M$. 

$M$ is said to be \emph{Hopf} if the Reeb vector field $\xi$ is principal.
It can be verify that $M$ is Hopf if and only if the integral curves of $\xi$ are geodesics in $M$.
The distribution $\mathcal D:=\ker\eta$ is known as the \emph{maximal holomorphic distribution}.

We called $M$ a \emph{$\mathfrak A$-principal} (resp. \emph{$\mathfrak A$-isotropic}) real hypersurface if 
the unit normal vector field $N$ is $\mathfrak A$-principal (resp. $\mathfrak A$-isotropic) everywhere.

We shall now fix some notations.
For any (local) section $A$ in $\mathfrak A^0$ and  vector field $X$ tangent to $M$, we denote by $V:=(AN)^T$, the tangential part of $AN$, 
$V^\circ=V-\eta(V)\xi$
and   
\begin{align}\label{eqn:B}
\left.\begin{aligned}
BX:=AX-\la V,X\ra N, \quad AN=V+fN \\
f:=\la AN,N\ra, \quad
g:=\la V,\xi\ra, \quad  
k:=|| V^\circ||.
\end{aligned}\right\}
\end{align}
We note that the entities $V$, $f$, $g$ and $k$ depend on the choice of $A$.
Following these notations, we have 
\begin{lemma}\label{lem:B}
\begin{enumerate}
	\item[(a)]$B\xi=-f\xi+\phi V$
	\item[(b)] 	$BV=-fV$
	\item[(c)] $B\phi V=(k^2+g^2)\xi+f\phi V-gV$
	\item[(d)] $B^2X=X-\la X,V\ra V$
	\item[(e)] $f^2+k^2+g^2=1$ 
	\item[(f)] $\trace B=-f$.
\end{enumerate}
\end{lemma}
\begin{proof}
It follows from $JA+AJ=0$ that $0=(JAN+AJN)^T=\phi V-f\xi-B\xi$. 
Since $A^2Z=Z$ for any vector $Z$ tangent to $Q^m$ and $\la V,V\ra=k^2+g^2$, 
the tangential and normal parts of $A^2N=N$ give (b) and (e) respectively.
For any $X$ tangent to $M$, $X=A^2X=B^2X+\la X,V\ra V$. This gives (d).
Next, with the help of (a) and (e), we can obtain (c) after putting $X=\xi$ in (d).
Finally,  (f) can be easily verified as  $\trace B=\trace A-\la AN,N\ra=-f$.
\end{proof}

For any $X$ tangent to $M$, we define 
\[\theta X:=JAX-\la X,B\xi\ra N.\]
By using the facts $JA+AJ=0$, $(JA)^2Z=Z$ for any $Z$ tangent to $Q^m$, we can also obtain the following identities
\begin{lemma}\label{lem:theta}
\begin{enumerate}
	\item[(a)]$\theta \xi=-V$
	\item[(b)] 	$\theta V=-(k^2+g^2)\xi-f\phi V$
	\item[(c)] $\theta \phi V=-f V-gB\xi$
	\item[(d)] $\theta^2X=X-\la X,B\xi\ra B\xi$
	\item[(e)] $\theta X=\phi BX-\la X,V\ra \xi=-B\phi X-\eta(X)V$
	\item[(f)] $\trace \theta=-g$.
\end{enumerate}
\end{lemma}
Next, we derive some identities arisen from the tangential and normal parts of (\ref{eqn:delA}).
\begin{lemma}\label{lem:B+theta}
\begin{enumerate}
	\item[(a)] $(\nabla_XB)Y=\la Y,V\ra SX+\la SX,Y\ra V+q(X)\theta Y$
	\item[(b)] $\nabla_XV=fSX-BSX+q(X)B\xi$
	\item[(c)] $Xf=-2\la X,SV\ra+gq(X)$
	\item[(d)] $(\nabla_X\theta)Y=\la Y,B\xi\ra S X+\la SX,Y\ra B\xi-q(X)BY$
	\item[(e)] $\nabla_XB\xi=gSX-\theta SX-q(X)V$
	\item[(f)] $Xg=-2\la SB\xi,X\ra-fq(X)$
	\item[(i)] $\nabla_XV^\circ=fSX-g\phi SX-BSX+2\la SB\xi,X\ra\xi+q(X)\phi V$
	\item[(j)] $\nabla_X\phi V=gSX+f\phi SX-\phi BSX-\la SV,X\ra\xi-q(X)V^\circ$.
\end{enumerate}
\end{lemma}
\begin{proof}
For any $X,Y$ tangent to $M$, we can obtain (a) and (b) from the tangential and normal parts of $(\hat\nabla_XA)Y=q(X)JAY$ respectively.
Next
\begin{align*}
Xf=&-X\la B\xi,\xi\ra=-\la\nabla_XB\xi,\xi\ra-2\la B\xi,\nabla_X\xi\ra
=-2\la X,SV\ra+gq(X).
\end{align*}
We observe that 
\[
(\hat\nabla_XJA)Y=(\hat\nabla_XJ)AY+J(\nabla_XA)Y=-q(X)AY.
\]
The tangential and normal parts give (d) and (e) respectively.
To obtain (f), we compute
\[
Xg=X\la V,\xi\ra=\la\nabla_XV,\xi\ra+\la V,\nabla_X\xi\ra=-2\la SB\xi,X\ra-fq(X).
\]
Finally, 
since $V^\circ=V-g\xi$ and $\nabla_X\phi V=(\nabla_X\phi)V+\phi\nabla_XV$,
by applying (b), (f), (\ref{eqn:delxi}), we can derive (i) and (j).
\end{proof}

\begin{lemma}\label{lem:g=0}
\begin{enumerate}
\item[(a)]
If $N$ is $\mathfrak A$-principal on a sufficiently small open set $\mathcal U\subset M$, then 
there exists a section $A$ of $\mathfrak A^0$ on $\mathcal U$ such that $f=1$.
\item[(b)] 
If  $N$ is not $\mathfrak A$-principal at $[z]$, then  there exist a sufficiently small neighborhood $\mathcal U$ of $[z]$ in $M$
and a section $A$ of $\mathfrak A^0$ on $\mathcal U$ such that $0\leq f<1$ and $g=0$.
\item[(c)] 
If  $N$ is $\mathfrak A$-isotropic on a sufficiently small open set $\mathcal U\subset M$, then 
there exists a section $A$  of $\mathfrak A^0$ on $\mathcal U$ such that $k=1$.
\end{enumerate}
\end{lemma}
\begin{proof}
Statement (a) is directly from the definition while the proof of Statement (b) can be found in \cite{berndt-suh0}.
Statement (c) is just a special case of Statement (b).
\end{proof}
For each $[z]\in M$, we define a subspace $\mathcal H^\perp$ of $T_{[z]}M$ by
\[
\mathcal H^\perp:=\vspan\{\xi, V, \phi V\}.
\]
Let $\mathcal H$ be the orthogonal complement of $\mathcal H^\perp$ in $T_{[z]}M$.
Then $\dim\mathcal H=2m-2$
when $N$ is $\mathfrak A$-principal at $[z]$ and $\dim \mathcal H=2m-4$ for otherwise.
 By virtue of Lemma~\ref{lem:B}, $B\mathcal H=\mathcal H$ and 
$B_{|\mathcal H}$ has two eigenvalues $1$ and $-1$.
For each $\varepsilon\in\{1,-1\}$, denote by $\mathcal H(\varepsilon)$ the eigenspace of  $B_{|\mathcal H}$ corresponding to $\varepsilon$.
Then  $\dim\mathcal H(\varepsilon)=m-1$ (resp. $\dim \mathcal H(\varepsilon)=m-2$) 
when $N$ is $\mathfrak A$-principal (resp. $N$ is not $\mathfrak A$-principal) at $[z]$. 
Moreover, we have 
$\phi \mathcal H(\varepsilon)=\mathcal H(-\varepsilon)$ by Lemma~\ref{lem:theta}(e).

\begin{lemma} \label{lem:pri+iso}
Let $M$ be a real hypersurface in $Q^m$. Then
\begin{enumerate}
	\item[(a)] if $M$ is  $\mathfrak A$-principal, then $S\mathcal H(-1)=0$;
  \item[(b)] if $M$ is  $\mathfrak A$-isotropic, then $SV=S\phi V=0$.
\end{enumerate}
\end{lemma}
\begin{proof}
Since $f=1$ when  $N$ is  $\mathfrak A$-principal everywhere, we have $k=g=0$ and $V=0$. 
It follows from Lemma~\ref{lem:B+theta}(b), (f) that 
\[
SX-BSX-2\la X,S\xi\ra \xi=0.
\]
By taking the transpose of this equation, we have 
\[
SX-SBX-2\la X,\xi\ra S\xi=0
\]
for any $X$ tangent to $M$. In particular, for $X\in\mathcal H(-1)$,  $2SX=0$ and so we obtain Statement (a).

Suppose $N$ is  $\mathfrak A$-isotropic everywhere. Then $f=g=0$ and $k=1$.
By Lemma~\ref{lem:B+theta}(c), (f) , we have $SV=0$ and $S\phi V=SB\xi=0$.
\end{proof}

It follows from (\ref{eqn:hatR}),   Lemma~\ref{lem:B} and Lemma~\ref{lem:theta} that the equations of Gauss and Codazzi equation  are  given by
\begin{align}
R(X,Y)=X\wedge Y+\phi X\wedge\phi Y-2\la\phi X,Y\ra\phi +BX\wedge BY+\theta X\wedge\theta Y+SX\wedge SY
\label{eqn:48} 
\end{align}
\begin{align}
(\nabla_X S)Y-(\nabla_YS)X=\eta(X)\phi Y-\eta(Y)\phi X-2\la\phi X,Y\ra\xi \nonumber\\
+\la X,V\ra BY-\la Y,V\ra BX+\eta(BX)\theta Y-\eta(BY)\theta X.  \label{eqn:50}
\end{align}
Let $\Ric $ be the Ricci tensor on $M$ and $h:=\trace S$. Then by (\ref{eqn:48}), we have 
\begin{align}\label{eqn:Ricci-tensor}
\Ric X=&(2m-1)X-3\eta(X)\xi+\la X,V\ra V+\la X,B\xi\ra B\xi	\nonumber\\
&-fBX -g\theta X-(S^2-h S)X.	
\end{align}


\section{Hopf  hypersurfaces in $Q^m$}
In this section, we assume that $M$ is a Hopf hypersurface in $Q^m$ with $\alpha=\la S\xi,\xi\ra$.
\begin{lemma}\label{lem:1}
Let $M$ be a Hopf hypersurface in $Q^m$. Then we have
\begin{align}\label{eqn:200}
\grad \alpha=&(\xi\alpha)\xi -2(f V^\circ+g \phi V) \\
(2S\phi S-\alpha(\phi S+S\phi)-2\phi)X=& -2\la X, V^\circ \ra\phi V+2\la X,\phi V\ra V^\circ \label{eqn:210} 
\end{align}
for any $X$ tangent to $M$.
\end{lemma}
\begin{remark}
This lemma can be obtained by a standard calculation using the Codazzi equation and has been proved in 
\cite{berndt-suh2}.
We will just outline the proof as below. 
\end{remark}
\begin{proof}[Proof of Lemma~\ref{lem:1}]
For any $X,Y$ tangent to $M$, we have 
\[
(\nabla_XS)\xi=(X\alpha)\xi+\alpha\phi SX-S\phi SX.
\]
By the Codazzi equation and this equation, we obtain
\begin{align*}
0
=&\la(\nabla_XS)Y-\nabla_YS)X,\xi\ra+2\la \phi X,Y\ra-2\la X,V\ra\la Y,B\xi\ra+2\la Y,V\ra\la X,B\xi\ra\\
=&(X\xi)\eta(Y)-(Y\xi)\eta(X)-2\la X,V\ra\la Y,B\xi\ra+2\la Y,V\ra\la X,B\xi\ra\\
&+\la \big(2\phi+\alpha(\phi S+S\phi)-2S\phi S\big) X,Y\ra.
\end{align*}
By substituting $Y=\xi$, we obtain (\ref{eqn:200}). By using (\ref{eqn:200}) and the above equation, we can get (\ref{eqn:210}). 
\end{proof}

By acting $\phi$ on both sides of (\ref{eqn:210}), we obtain
\begin{align*}
(2\phi S\phi S+\alpha S-\alpha\phi S\phi +2)X-(\alpha^2+2)\eta(X)\xi=2\la X,V^\circ\ra V^\circ+2\la X,\phi V\ra\phi V.
\end{align*} 
This  implies that $(\phi S\phi) S-S(\phi S\phi)=0$.
Hence there exists a local orthonormal frame 
$\{X_0 =\xi,X_1, . . . , X_{m-1}, X_{m}=\phi X_1,\cdots, X_{2m-2}=\phi X_{m-1}\}$
 such that 
\begin{align}\label{eqn:800}
SX_j =\lambda_jX_j, \quad \phi S\phi X_j = -\mu_jX_j;  \quad j\in\{1,\cdots,m-1\}.
\end{align}
By using (\ref{eqn:210}) and (\ref{eqn:800}), we get  
\begin{align}\label{eqn:810}
\{-2\lambda_j\mu_j+\alpha(\lambda_j+\mu_j)+2\}X_j=2\la X_j, V^\circ\ra V^\circ+2\la X_j,\phi V\ra\phi V
\end{align} 
for $j\in\{1,\cdots,m-1\}$.
If  $N$ is $\mathfrak A$-principal, then $\mathcal D=\mathcal H$ and hence $S\mathcal H\subset\mathcal H$.
On the other hand,  $k>0$ or $V^\circ\neq0$ when $N$ is not $\mathfrak A$-principal. 
Hence, there is exactly one $j$'s, say $j=1$, such that $-2\lambda_1\mu_1+\alpha(\lambda_1+\mu_1)+2\neq0$.
This means that $\mathcal H$ is spanned by the vectors $X_2,\cdots,X_{m-1},\phi X_2,\cdots,\phi X_{m-1}$;
so $S\mathcal H\subset\mathcal H$ and 
\begin{align*}
-2\lambda_j\mu_j+\alpha(\lambda_j+\mu_j)+2=0, \quad j\in\{2,\cdots,m-1\}
\end{align*} 
in this case.
Furthermore, by selecting an appropriate local section $A$, we can set $X_1=(1/k)V^\circ$ and 
\begin{align*}
-2\lambda_1\mu_1+\alpha(\lambda_1+\mu_1)+2-2k^2=0.
\end{align*} 

We have shown the following lemma.
\begin{lemma}\label{lem:7b}
Let $M$ be a Hopf hypersurface in $Q^m$, $m\geq 3$. Then 
 $S\mathcal H\subset\mathcal H$. 
 If  $E$ is a vector tangent to $\mathcal H$ such that $SE=\lambda E$ and $S\phi E=\mu\phi E$, then 
										\begin{align*}
-2\lambda\mu+\alpha(\lambda+\mu)+2=0.
\end{align*} 
Furthermore, if $N$ is not  $\mathfrak A$-principal, then there exists a local section $A$ of $\mathfrak A^0$ such that 
$SV^\circ =tV^\circ$ and  $S\phi V=\omega\phi V$, where $t$ and $\omega$ satisfy 
 \begin{align*}
-2t\omega+\alpha(t+\omega)+2-2k^2=0.
\end{align*} 
\end{lemma}

\begin{lemma}\label{lem:contact}
Let $M$ be a real hypersurface in $Q^m$, $m\geq3$. Then $\phi S+S\phi\neq0$ on every open set $\mathcal U\subset M$. 
\end{lemma}
\begin{proof}
Suppose $\phi S+S\phi=0$ on $\mathcal U$.  
It is clear that $\xi$ is principal  at each $[z]\in \mathcal U$.
Since $\dim \mathcal H\geq 2m-4>0$, we take a principal vector $X\in\mathcal H$ in line with Lemma~\ref{lem:7b}.
It follows that $\lambda+\mu=0$ and so $2\lambda^2+2=0$.
This is a contradiction and we obtain the Lemma.  
\end{proof}

\begin{lemma}\label{lem:7}
Let $M$ be a Hopf hypersurface in $Q^m$. Then $\alpha$ is constant if and only if either
$M$ is $\mathfrak A$-principal or $M$ is $\mathfrak A$-isotropic.
\end{lemma}
\begin{proof}
Suppose $\alpha$ is a constant. 
Then by (\ref{eqn:200}), we have $fV^\circ+g\phi V=0$ and so $fk=gk=0$.
Let 
\begin{align*}
M_1=\{[z]\in M : N_{[z]} \text{ is  $\mathfrak A$-principal}\}.
\end{align*}
If $N$ is not $\mathfrak A$-principal everywhere, it follows from Lemma~\ref{lem:g=0} that $k\neq0$ on 
$M_1^c$, which implies that 
$f=g=0$ on $M_1^c$
and hence $N$ is $\mathfrak A$-isotropic on $M_1^c$.

Now consider the function $F:=f^2+g^2$. We note that $F$ is independent of the choice of $A\in\mathfrak A^0$ and globally defined on $M$. 
Then $F=1$ on $M_1$  and $F=0$ on $M_1^c$.
By the continuity of $F$, $M=M_1^c$ and so it is $\mathfrak A$-isotropic.   

Conversely, we have two cases: $M$ is $\mathfrak A$-principal and $M$ is $\mathfrak A$-isotropic.
If  $M$ is $\mathfrak A$-principal, then $f=1$,  $g=0$ and  $V=0$.
On the other hand, we have $f=g=0$ when $M$ is $\mathfrak A$-isotropic.
By using (\ref{eqn:200}), we deduce that 
$\grad\alpha=(\xi\alpha)\xi$ in both cases.
It follows that 
\[
(XY-\nabla_XY)\alpha=(X\xi\alpha)\eta(Y)+(\xi\alpha)\la Y,\phi SX\ra.
\]
Hence
\[
0=(X\xi\alpha)\eta(Y)-(Y\xi\alpha)\eta(X)+(\xi\alpha)\la Y,(\phi S+S\phi)X\ra.
\]
Substituting $Y=\xi$ gives $X\xi\alpha=(\xi\xi\alpha)\eta(X)$. Hence $(\xi\alpha)(\phi S+S\phi)=0$.
It follows from Lemma~\ref{lem:contact} that $ \phi S+S\phi\neq0$ on a dense open subset of $M$. 
Hence $\xi\alpha=0$  by its continuity and so $\grad \alpha=0$.
Accordingly, $\alpha$ is a constant.
\end{proof}

\begin{lemma}
Assuming the notation and hypotheses in Lemma~\ref{lem:7b}, 
if $M$ is neither $\mathfrak A$-principal nor $M$ is $\mathfrak A$-isotropic, then
\begin{align}
(\xi\alpha)(\lambda+\mu)=&-2g(\lambda-\mu)\la BE,E\ra-2f(\lambda-\mu)\la BE,\phi E\ra \label{eqn:a-40}\\
(\xi\alpha)(t+\omega)=&2fg(t-\omega) \label{eqn:a-50} \\
\grad (\xi\alpha)=& (\xi\alpha)\xi+2f(\omega-\alpha)\phi V-2g(t-\alpha)V^\circ. \label{eqn:a-20}
\end{align} 
\end{lemma}
\begin{proof}
By using (\ref{eqn:200}), we have 
\begin{align*}
(XY-\nabla_XY)\alpha
=&(X\xi\alpha)\eta(Y)-2(Xf)\la V^\circ,Y\ra-2(Xg)\la \phi V,Y\ra\\
&+(\xi\alpha)\la\nabla_X\xi,Y\ra-2f\la\nabla_XV^\circ,Y\ra-2g\la\nabla_X\phi V,Y\ra.
\end{align*}
It follows that 
\begin{align*}
0=&(X\xi\alpha)\eta(Y)-2(Xf)\la V^\circ,Y\ra-2(Xg)\la \phi V,Y\ra    \notag  \\
     &-(Y\xi\alpha)\eta(X)+2(Yf)\la V^\circ,X\ra+2(Yg)\la \phi V,X\ra \notag\\
 &+(\xi\alpha)\la\nabla_X\xi,Y\ra-2f\la\nabla_XV^\circ,Y\ra-2g\la\nabla_X\phi V,Y\ra \notag\\
& -(\xi\alpha)\la\nabla_ Y\xi,X\ra+2f\la\nabla_YV^\circ,X\ra+2g\la\nabla_Y\phi V,X\ra \notag\\
=&\{X(\xi\alpha)+2g(t-2\alpha)\la V^\circ,X\ra-4f(\omega-\alpha)\la\phi  V,X\ra\}\eta(Y)									\notag\\
& -\{Y(\xi\alpha)+2g(t-2\alpha)\la V^\circ,Y\ra-4f(\omega-\alpha)\la\phi  V,Y\ra\}\eta(X)									\notag\\
&+\la (\xi\alpha)(\phi S+S\phi)X+2g(\phi BS+SB\phi)X+2f\la(SB-BS)X,Y\ra.			
\end{align*}
for any $X,Y\in TM$. 
In particular, if $X=E$ and $Y=\phi E$, then 
we get  (\ref{eqn:a-40}).
On the other hand,  (\ref{eqn:a-50}) can be obtained by putting $X=V^\circ$ and $Y=\phi V$ in preceding equation.
Finally, letting $X=\xi$, gives  (\ref{eqn:a-20}).
\end{proof}


\section{Tubes around $Q^{m-1}$ in $Q^m$}
\label{sec:tube}
The totally geodesic complex hypersurface $Q^{m-1}$ in $Q^m$ is determined by the equations
\[
z_0^2+\cdots+z_m^2=0, \quad z_{m+1}=0.
\]
$Q^{m-1}$ is a singular orbit of the cohomogeneity one action $SO_{m+1}\subset SO_{m+2}$ on $Q^m$.
The other singular orbit is a totally geodesic totally real $m$-dimensional sphere $S^m=SO_{m+1}/SO_m$.

The distance between the two singular orbits of the $SO_{m+1}$-action is $\pi/2\sqrt2$ and   
each principal orbit of the action is a tube of radius 
$r\in]0,\pi/2\sqrt2[$ around the totally geodesic $Q^{m-1}\subset Q^m$.
A principal orbit of the action is a homogeneous space of the form $SO_{m+1}/S_{m-1}$
which is a $S^1$-bundle over $Q^{m-1}$, and a $S^{m-1}$-bundle over $S^m$.

From the construction of $\mathfrak A$ it is clear that $T_{[z]}Q^{m-1}$ and $T_{[z]}^\perp Q^{m-1}$ are $A$-invariant 
for each $A\in\mathfrak A^0_{[z]}$.
Moreover, since the real dimensional of $Q^{m-1}$ in $Q^m$ is 2, for each unit vector 
$\zeta\in T_{[z]}^\perp Q^{m-1}$, $[z]\in Q^{m-1}$,
there exists $A\in\mathfrak A^0_{[z]}$ such that $A\zeta=\zeta$ and so $AJ\zeta=-J\zeta$.
Hence
\[
T_{[z]}Q^{m-1}=(V(A)\ominus\mathbb R\zeta)\oplus \ J(V(A)\ominus\mathbb R\zeta).
\]  
It follows that the Jacobi operator $\hat R_\zeta:=\hat R(\cdot,\zeta)\zeta$ is given by
\[
\hat R_\zeta Y=Y+AY-2\la Y,\zeta\ra\zeta+2\la Y, J\zeta\ra J\zeta.
\]
It has two constant eigenvalues, $0$ and $2$ with corresponding eigenspaces
$J(V(A)\ominus\mathbb R\zeta)\oplus \mathbb R\zeta$ and 
$(V(A)\ominus\mathbb R\zeta)\oplus\mathbb RJ\zeta$.

We will use the standard Jacobi field method to determine the principal curvatures and their corresponding eigenspaces of a tube around a totally geodesic $Q^{m-1}$ in $Q^m$.

Fixed $r\in]0,\pi/2\sqrt2[$.
For each $[z]\in Q^{m-1}$ and unit vector $\zeta\in T_{[z]}^\perp Q^{m-1}$, 
denote by $\gamma_{\zeta}(s)$ the unit speed geodesic in $Q^m$ passes through $[z]$ at $s=0$ with intial velocity $\zeta$.

Let $\mathcal Y$ be the Jacobi field along $\gamma_{\zeta}$ with initial values 
$\mathcal Y(0)\in T_{[z]}Q^{m-1}$ and $\dot{\mathcal Y}(0)+S_\zeta \mathcal Y(0)=\dot{\mathcal Y}(0)\in T_{[z]}^\perp Q^{m-1}$, where $S_{\zeta}$ denotes the shape operator of $Q^{m-1}$ with respect to $\zeta$.
Then $\dot\gamma_{\zeta}(r)$ is a unit vector normal to the tube $M_r$ of radius $r$ around $Q^{m-1}$ at $\gamma_{\zeta}(r)$
and the tangent space of $M_r$ at $\gamma_{\zeta}(r)$ is spanned by 
$\mathcal Y(r)$. Moreover the shape operator $S$ of $M_r$ with respect to $N=-\dot\gamma_\zeta(r)$ can be determine by the eqution (cf. \cite[pp.225]{berndt-console-olmos})
\[
S\mathcal Y(r)=\dot{\mathcal Y}(r).
\]

To determine the principal curvatures of $M_r$ and their corresponding eigenspaces, we consider the following Jacobi field
\begin{align*}
\mathcal Y_X(t)=\left\{\begin{array}{rl}
(1/\sqrt2t)\sin (\sqrt2t)\mathcal E_X(t), & X=J\zeta \\
(1/\sqrt2t)\cos (\sqrt2t)\mathcal E_X(t), & X\in   V(A)\ominus\mathbb R\zeta \\
\mathcal E_X(t),                          & X\in J(V(A)\ominus\mathbb R\zeta)
\end{array}\right.
\end{align*}
where $\mathcal E_X$ is the parallel vector field along $\gamma_{\zeta}$ with $\mathcal E_X(0)=X$.
It follows that $M_r$ has three constant principal curvatures 
$\sqrt2\cot(\sqrt2r)$, $-\sqrt2\tan(\sqrt2r)$ and $0$, with 
eigenspaces $\mathbb RJ\zeta$, $V(A)\ominus\mathbb R\zeta$ and 
$J(V(A)\ominus\mathbb R\zeta)$ respectively, of which we have identified the subspaces obtained by
parallel translation along $\gamma_{\zeta}$ from $[z]$ to $\gamma_{\zeta}(r)$.

We can see that the unit vector $N$ for $M_r$ is $\mathfrak A$-principal and the shape operator $S$ satisfies
$\phi S+S\phi=-\sqrt2\tan(\sqrt2r)\phi$. We summarize these observations in the following theorem.

\begin{theorem}[\cite{berndt1}]\label{thm:tube}
Let $M$ be the tube of radius $r\in]0,\pi/2\sqrt2[$ around the totally geodesic $Q^{m-1}$ in $Q^m$.
Then the normal bundle of $M$ consists of $\mathfrak A$-principal singular tangent vectors of $Q^m$, and 
$M$ has three constant principal curvatures 
\[
\alpha=\sqrt{2}\cot(\sqrt2r), \quad \lambda=-\sqrt{2}\tan(\sqrt2r), \quad \mu=0.
\] 
The corresponding eigenspaces are 
\[
T_{\alpha}=\mathbb RJN, \quad T_{\lambda}=V(A)\ominus\mathbb RN, \quad
T_{\mu}=J(V(A)\ominus\mathbb RN) 
\]
and the corresponding multiplicities are
\[
m(\alpha)=1, \quad m(\lambda)=m-1=m(\mu)
\] 
where $A$ is a conjugation such that $AN=N$ and $N$ is a unit vector normal to $M$.
Further, the shape operator $S$ satisfies
$\phi S+S\phi=-\sqrt2\tan(\sqrt2r)\phi$.
\end{theorem}

Theorem~\ref{thm:tube} tells us that a tube around a totally geodesic $Q^{m-1}$ in $Q^m$ is Hopf and $\mathfrak A$-principal.
We shall show that the converse is also true. 
\begin{theorem}\label{thm:main}
Let $M$ be a Hopf hypersurface of the complex quadric $Q^m$, $m\geq3$. 
Then $M$ is $\mathfrak A$-principal if and only if 
$M$ an open part of a tube around a totally geodesic $Q^{m-1}$ in $Q^m$.
\end{theorem}
\begin{proof}
Suppose $M$ is Hopf and $\mathfrak A$-principal.
For each $[z]\in M$, 
since $S\mathcal H(-1)=0$ and $\phi \mathcal H(-1)=\mathcal H(1)$, after putting $X\in\mathcal H(-1)$ 
in (\ref{eqn:210}), we have $\alpha S\phi X=-2\phi X$
of which implies that $\alpha\neq0$ and $S\phi X=-(2/\alpha)\phi X$.
By Lemma~\ref{lem:7}, $\alpha$ is a constant, without loss of generality, 
we put $\alpha=\sqrt2\cot(\sqrt2r)$ with $0<r<\pi/2\sqrt2$.
Hence we see that $M$ has three constant principal curvatures:
\[
\alpha=\sqrt{2}\cot(\sqrt2r), \quad \lambda=-\frac{2}{\alpha}=-\sqrt{2}\tan(\sqrt2r), \quad \mu=0.
\] 
The corresponding principal curvature spaces are 
\[
\mathcal T_{\alpha}=\mathbb R\xi, \quad 
\mathcal T_{\lambda}=\mathcal H(1), \quad
\mathcal T_{\mu}=\mathcal H(-1) 
\]
and the corresponding multiplicities are
\[
m(\alpha)=1, \quad m(\lambda)=m-1=m(\mu).
\]

We will use the Jacobi field method again to determine the focal submanifold of $M$. 
As in Section~\ref{sec:tube},
denote by $\gamma_{N}(s)$ is the unit speed geodesic in $Q^m$ passes through $[z]\in M$ at $s=0$ with initial velocity $N_{[z]}$. Since $M$ is a real hypersurface in $Q^m$, we may identify the unit normal bundlle $B(M)$ as $M$, and the focal map $\Phi_r([z])=\gamma_{N}(r)$.

Let $\mathcal Y_X$ be the Jacobi field along $\gamma_{N}$ with initial values 
$\mathcal Y_X(0)=X\in T_xM$ and $\dot{\mathcal Y}_X(0)=-SX$.
Then 
\[
d\Phi_r(\sigma)X=\mathcal Y_X(r).
\]

As $N$ is $\mathfrak A$-principal, by using (\ref{eqn:hatR}), the normal Jacobi operator 
$R_N:=\hat R(\cdot,N)N$ is given by
\[
R_NY=Y+BY+2\eta(X)\xi.
\] 
It follows that $R_N$ has two constant eigenvalues $0$, $2$ with corresponding eigenspaces
$\mathcal T_\mu$ and $\mathcal T_\lambda\oplus T_\alpha$ respectively.

To compute $d\Phi_r([z])X$, $X\in T_{[z]}M$,
we select the Jacobi field
\begin{align}\label{eqn:jacobi}
\mathcal Y_X(t)=\left\{\begin{array}{rl}
\left(\cos (\sqrt2t)-({\alpha}/{\sqrt{2}})\sin (\sqrt2t)\right)\mathcal E_X(t), & X=\xi \\
\left(\cos (\sqrt2t)-({\lambda}/{\sqrt{2}})\sin (\sqrt2t)\right)\mathcal E_X(t), & X\in T_\lambda \\
\mathcal E_X(t), & X\in T_\beta
\end{array}\right.
\end{align}
where $\mathcal E_X$ is the parallel vector field along $\gamma_{[z]}$ with $\mathcal E_X(0)=X$.
Then we have $d\Phi_r([z])X=\mathcal Y_X(r)=0$ if and only if $X=\xi$ and conclude that $\Phi_r$ has constant rank $2m-2$.
It follows that $\Phi_r$ is locally a submersion onto a submanifold $\tilde M$ in $Q^m$ of real dimension $2m-2$. 

Note that $T_\lambda\oplus T_\mu=\mathcal D_{[z]}$ is invariant under $J$, $J$ is invariant under parallel translation along geodesics and the tangent space $T_{\Phi_r([z])}\tilde M$  of $\tilde M$ at $\Phi_r([z])$  is obtained by parallel translation of $T_\lambda\oplus T_\beta$ along the geodesic $\gamma_{[z]}$, we see that $\tilde M$ is a complex $(m-1)$-dimensional complex submanifold in $Q^m$, that is, a complex hypersurface.
 
Now we claim that $\tilde M$ is totally geodesic.
To prove this claim, we note that the vector $\zeta=\dot\gamma_{N}(r)$ is a unit normal vector of $\tilde M$ at $\Phi_r([z])$ and the shape operator $\tilde S_\zeta$ of $\tilde M$ in $Q^m$ with respect to $\zeta$ can be determined by
$\tilde S_\zeta X=-\dot{\mathcal Y}_X(r)$, where $X\in T_\lambda\oplus T_\mu$ and $\mathcal Y_X$ is the Jacobi fields given by (\ref{eqn:jacobi}).
First, it is clear that $\dot{\mathcal Y}_X(r)=0$ for $X\in T_\beta$. 
Next, as $\lambda=-\sqrt{2}\tan(\sqrt2r)$ we see that $\dot{\mathcal Y}_X(r)=0$ for $X\in T_\lambda$.
Hence, $\tilde M$ is a totally geodesic complex hypersurface in $Q^m$. 

By the rigidity of totally geodesic submanifolds, $M$ is an open part of a tube of radius $r$ around a  connected, complete, totally geodesic complex hypersurface $\tilde M$ of $Q^m$.
According to the classification of totally geodesic submanifolds in $Q^m$ (cf. \cite{klein}), 
$\tilde M$ is the totally geodesic 	complex hypersurface $Q^{m-1}$ in $Q^m$. This implies that $M$ is locally congruent to a tube around $Q^{m-1}$ in $Q^m$.
\end{proof}

\section{Contact real hypersurfaces in $Q^m$}
Let $M$ be a real hypersurface in a K\"ahler manifold $\hat M$. 
Denote by $\Phi(\cdot,\cdot):=\la\cdot,\phi~\cdot\ra$ the fundamental $2$-form.
If there exist a non-zero function $\rho$ on $M$ such that $d\eta=\rho\Phi$, then $M$  admits a contact structure.
In this case, we called $M$ a \emph{contact real hypersurface} in $\hat M$.
Since $d\eta(X,Y)=\la (\phi S+S\phi)X,Y\ra$, a real hypersurface $M$ in $\hat M$ is contact if and only if 
\begin{align}\label{eqn:contact1}
\phi S+S\phi=\rho\phi
\end{align}
for some non-zero function $\rho$ on $M$. 


\begin{theorem}\label{thm:contact}
Let $M$ be a real hypersurface in $Q^m$. Then $M$ is contact if and only if 
$M$ is an open part of a tube around a totally geodesic $Q^{m-1}$ in $Q^m$.
\end{theorem}
\begin{proof}
Suppose $M$ is a contact real hypersurface, thas is, it satisfies (\ref{eqn:contact1}). Then it is clear that $M$ is Hopf.
Furthermore, $\rho$ must be a nonzero constant (cf. \cite{berndt-suh0}). 
We first consider the case $M$ is neither $\mathfrak A$-isotropic nor $\mathfrak A$-principal.
Then there exists an open subset $\mathcal U\subset M$ on which $0<k<1$.
Without loss of generality, we assume $\mathcal U=M$. 

Let $\lambda$, $\mu$, $t$, $\omega$ and $E$ be as stated in Lemma~\ref{lem:7b}. 
We can assume that $E$ is a unit vector. By the Codazzi equation, we have 
\begin{align}
0=&\la(\nabla_ES)\phi E-(\nabla_{\phi E}S)E,V^\circ\ra \notag\\
=&\la(t\mathbb I-S)\nabla_EV^\circ,\phi E\ra-\la(t\mathbb I-S)\nabla_{\phi E}V^\circ, E\ra \notag\\
=&-tg(\lambda+\mu)-t(\lambda-\mu)\la BE,\phi E\ra+2g\lambda\mu. \label{eqn:a-100}
\end{align}
Similarly, we compute
\begin{align}
0=&\la(\nabla_ES)\phi E-(\nabla_{\phi E}S)E,\phi V\ra \notag\\
=&\la(t\mathbb I-S)\nabla_EV^\circ,\phi E\ra-\la(t\mathbb I-S)\nabla_{\phi E}V^\circ, E\ra \notag\\
=&\omega f(\lambda+\mu)-\omega(\lambda-\mu)\la BE,E\ra-2f\lambda\mu. \label{eqn:a-110}.
\end{align}
It follows from (\ref{eqn:a-40})--(\ref{eqn:a-50}) and (\ref{eqn:a-100})--(\ref{eqn:a-110}) that 
$
(\xi\alpha)\{(\lambda+\mu)t\omega-(t+\omega)\lambda\mu\}=0.
$
By applying Lemma~\ref{lem:7b} and the fact that $\lambda+\mu=\omega+t=\rho$, we obtain
$(\xi\alpha)\rho k^2=0$ and hence 
$
\xi\alpha=0.
$

Since $V^\circ$ and $\phi V$ are orthogonal, we obtain $g(t-\alpha)=f(\omega-\alpha)=0$
by (\ref{eqn:a-20}).
If $fg\neq0$, then $t=\omega=\alpha$. But these imply that $\alpha=(t+\omega)/2=\rho/2$ is a constant; a contradiction to 
Lemma~\ref{lem:7}.
Hence we have either $f=0$ or $g=0$.
Without loss of generality, we assume $g=0$, hence $f\neq0$ and $\alpha=\omega=\rho-t$. 
By substituting these into the second equation in  Lemma~\ref{lem:7b}, give
$-2\alpha^2+\rho\alpha-2+2k^2=0$.
By applying Lemma~\ref{lem:B+theta}(c) and (\ref{eqn:200}), we see that 
\begin{align*}
0=(-4\alpha+\rho)\grad\alpha+2\grad(k^2)=-2f(-4\alpha+\rho)V-2\grad(f^2)=-3\rho.
\end{align*}
This is a contradiction. 
Consequently, $M$ is either $\mathfrak A$-isotropic or  $\mathfrak A$-principal.
It is clear that  $N$ is  not $\mathfrak A$-isotropic everywhere 
(for otherwise,  we have $2\rho\phi V=(\phi S+S\phi)V=0$ by virtue of  Lemma~\ref{lem:pri+iso}, which is impossible).
Hence  $M$ is $\mathfrak A$-principal.
According to Theorem~\ref{thm:main}, $M$ is an open part of a tube around a totally geodesic $Q^{m-1}$ in $Q^m$.

Conversely, as shown in Theorem~\ref{thm:tube}, the shape operator of a tube of radius $r$ around a totally geodesic $Q^{m-1}$ in $Q^m$ satisfies $\phi S+S\phi=-\sqrt{2}\tan(\sqrt2r)\phi $. Hence, it is contact and this completes the proof.
\end{proof}

\begin{remark}
Contact real hypersurfaces in K\"ahler manifolds with constant mean curvature were studied in \cite{berndt-suh0}.
\end{remark}

Next, we study real hypersurfaces $M$ in $Q^m$ under a weaker version of (\ref{eqn:contact1}), i.e., 
\begin{align} \label{eqn:B:D}
\phi(\phi S+S\phi-\rho\phi)\phi=0,
\end{align}
for some function $\rho$ on $M$.
We shall first derive some identities from the condition  (\ref{eqn:B:D}).
Note that (\ref{eqn:B:D}) is equivalent to
\begin{align*}
\la(\phi S+S\phi - \rho\phi)Y,Z\ra=0
\end{align*}
for any vector fields $Y$ and $Z$ in $\mathcal D$.
Differentiating this equation covariantly in the direction of $X$  in $\mathcal D$ we get
\begin{align*}
   \la\phi SY,\nabla_XZ\ra+\la(\nabla_X\phi)SY+\phi(\nabla_XS)Y+\phi S\nabla_XY,Z\ra	 &	\\ \nonumber
+\la S\phi Y,\nabla_XZ\ra+\la(\nabla_XS)\phi Y+S(\nabla_X\phi) Y+S\phi\nabla_XY,Z\ra	&\\ \nonumber
- d\rho(X)\la\phi Y,Z\ra-\rho\la\phi Y,\nabla_XZ\ra-\rho\la(\nabla_X\phi)Y+\phi\nabla_XY,Z\ra&=0.
\end{align*}
By using (\ref{eqn:delxi}) and (\ref{eqn:B:D}), this equation can be reformed as
\begin{align*}
-\la Z,\phi S\xi\ra\la\phi SX,Y\ra+\la Y,\phi S\xi\ra\la \phi SX,Z\ra
-\la(\nabla_XS)Y,\phi Z\ra+\la(\nabla_XS)Z,\phi Y\ra \\ \nonumber
		+\eta(SY)\la SX,Z\ra-\eta(SZ)\la SX,Y\ra - d\rho(X)\la\phi Y,Z\ra=0.
\end{align*}
Now by replacing $X,Y$ and $Z$ cyclically in the above equation and then summing these equations, with the help of the Codazzi equation 
Lemma~\ref{lem:theta}(e) and (\ref{eqn:B:D}), we obtain
\begin{align*}
\mathfrak S(\rho\la X,\phi S\xi\ra+d\rho(X))\la\phi Y,Z\ra=0
\end{align*}
where $\mathfrak S$ denotes the cyclic sum over $X,Y$ and $Z$.
Let $X$ be an arbitrary vector in $\mathcal D$. 
Since $m\geq3$, we may $Y\perp X,\phi X$ and $Z=\phi Y$ in the above equation then 
\[
\rho\la X,\phi S\xi\ra+d\rho(X)=0
\]
for any $X$ in $\mathcal D$.


In a special case where $\rho$ is a non-zero constant, The above equation implies that $\phi S\xi=0$ which means $\xi$ is principal and  
so $(\phi S+S\phi - \rho\phi)\xi=0$. 
Consequently, we have $\phi S+S\phi - \rho\phi=0$, for some non-zero constant $\rho$, and hence it follows from 
Theorem~\ref{thm:contact} that we obtain

\begin{theorem} \label{thm:contact1} 
Let $M$ be a real hypersurface in $Q^m$, $n\geq3$. Then $M$ satisfies 
\begin{align*}
\phi(\phi S+S\phi-\epsilon\phi)\phi=0
\end{align*}
for some constant $\epsilon\neq0$ if and only if $M$ is an open part of a tube around a totally geodesic $Q^{m-1}$ in $Q^m$.
\end{theorem}

\begin{remark}
Theorem \ref{thm:contact1} was proved in \cite{kon-loo} for real hypersurfaces in non-flat complex space forms.
\end{remark}

\section{Pseudo-Einstein real hypersurfaces in $Q^m$}
Suppose that $M$ is pseudo-Einstein, that is, 
\begin{align}\label{eqn:Einstein}
\Ric X=&aX+b\eta(X)\xi	
\end{align}
where $a, b$ are constants.
By (\ref{eqn:Ricci-tensor}), we see that $M$ is pseudo-Einstein if and only if 
\begin{align}\label{eqn:P}
PX=(2m-a-1)X-(3+b)\eta(X)\xi+\la X,V\ra V+\la X,B\xi\ra B\xi-fBX-g\theta X
\end{align}
where $P:=S^2-hS$. 
\begin{lemma}\label{lem:schur-2}
Let $M$ be a pseudo-Einstein real hypersurface in $Q^m$, $m\geq 3$. If $b\neq0$, then $M$ is Hopf. 
\end{lemma}
\begin{proof}
It follows from the hypothesis (\ref{eqn:Einstein}) that
\begin{align*}
(\nabla_X\Ric)Y=&b\{\la \phi SX,Y\ra\xi+\eta(Y)\phi SX\}.	
\end{align*}
Take an orthonormal basis $\{e_1,\cdots,e_{2m-1}\}$ on $T_{[z]}M$. Then
\begin{align*}
X(\trace \Ric )=\sum^{2m-1}_{j=1}\la(\nabla_X\Ric )e_j,e_j\ra=0\\
\diver \Ric (X)=\sum^{2m-1}_{j=1}\la(\nabla_{e_j}\Ric )X,e_j\ra=b\la\phi S\xi,X\ra.
\end{align*}
By the well-known formula $d(\trace \Ric )=2\diver \Ric $, we obtain
$b\phi S\xi=0$. Hence we conclude that $M$ is Hopf if $b\neq0$.
\end{proof}

\begin{theorem}\label{thm:Eintein0}
Let $M$ be a pseudo-Einstein real hypersurface in $Q^m$, $m\geq3$. Then $M$ is either $\mathfrak A$-principal or
$\mathfrak A$-isotropic.
\end{theorem}
\begin{proof}
Suppose $M$ is neither $\mathfrak A$-principal nor
$\mathfrak A$-isotropic. Then there exists an open subset $\mathcal U\subset M$ on which $0<f<1$ and  $g=0$.
Without loss of generality. We assume $\mathcal U=M$. 
It follows from (\ref{eqn:P}) that 
$P\xi=(2m-a-2-2k^2-b)\xi-2f\phi V$.
If $b\neq0$, then $M$ is Hopf by Lemma~\ref{lem:schur-2} and so
$(\alpha^2-h\alpha)\xi= (2m-a-2-2k^2-b)\xi-2f\phi V$ which implies that $f=0$; a contradiction. Hence we have $b=0$.
It follows that $P$ has at most five distinct eigenvalues 
\begin{eqnarray*}
&	\sigma_0=2m-a, \quad
	\sigma_1=2m-a-1-f, \quad 
	\sigma_2=2m-a-1+f,& \\
&	\sigma_3=2m-a-2-2k,\quad 
	\sigma_4=2m-a-2+2k&
\end{eqnarray*}
with eigenspaces 
\begin{eqnarray*}
&	\mathcal T_0=\mathbb R V, \quad
	\mathcal T_1=\mathcal H(1), \quad 
	\mathcal T_2=\mathcal H(-1), \quad 
	\mathcal T_3=\mathbb R W_3,\quad 
	\mathcal T_4=\mathbb R W_4&
\end{eqnarray*}
where
\begin{align*}
W_3=&r\xi+\frac sk\phi V, \quad W_4=-s\xi+\frac rk\phi V, \quad 
r=\sqrt{1+k}, \quad s=\sqrt{1-k}.
\end{align*}

Since $f,k>0$ and $f^2+k^2=1$, we can easily verify the following 
\begin{align}\label{eqn:2=4}
\left. \begin{array}{l}
\sigma_0\notin\{\sigma_1,\sigma_2, \sigma_3,\sigma_4\}\\
\sigma_1\notin\{\sigma_0,\sigma_2, \sigma_3,\sigma_4\}\\
\sigma_3\notin\{\sigma_0,\sigma_1, \sigma_2,\sigma_4\}
.
\end{array}\right\}
\end{align}
Since $PS=SP$, we conclude that 
\begin{align*}
S\mathcal H(1)\subset \mathcal H(1), \quad 
SV=tV, \quad SW_3=\kappa W_3
\end{align*} 
where $t$ and $\kappa$ are functions satisfying
\begin{align}\label{eqn:eigen-3}
t^2-ht=\sigma_0, \quad \kappa^2-h\kappa=\sigma_3.
\end{align}
Let 
\[
U:=SW_4-\tau W_4\in\mathcal H(-1), \quad \tau=\frac12\la SW_4,W_4\ra.
\]
A straightforward calculation gives
\begin{align}\label{eqn:basic}
\left.\begin{array}{ll}
BW_3=W_4, & \\
\grad r=stk^{-1}V, &  \grad s=-rtk^{-1}V\\
\grad (sk^{-1})=-r(2-k)tk^{-3}V, & \grad (rk^{-1})=-s(2+k)tk^{-3}V\\
(\phi S+S\phi)W_3=-s(\kappa+t)k^{-1}V, & (\phi S+S\phi)W_4=-r(\tau+t)k^{-1}V+\phi U\\
S\phi SW_3=-s\kappa tk^{-1}V, & S\phi SW_4=-r\tau tk^{-1}V+S\phi U\\
\phi BSW_3=-r\kappa k^{-1} V, & \phi BSW_4=-s\tau k^{-1} V-\phi U\\
SB\phi W_3=sft k^{-1} V,      & SB\phi W_4=rft k^{-1} V.
\end{array}\right\}
\end{align}
By applying (\ref{eqn:basic}), we compute
\begin{align*}
(\nabla_XS)W_3
=&(X\kappa)W_3+(\kappa\mathbb I-S)\nabla_XW_3\\
=&(X\kappa)W_3-\frac{r(2-k)t}{k^{3}}\la X,V\ra(\kappa\mathbb I-S)\phi V+\frac{2(\kappa-t)}{rk}\la X,SB\xi\ra V\\
&+\frac rk(\kappa\mathbb I-S)\phi SX-\frac sk(\kappa\mathbb I-S)\phi BSX.
\end{align*}
By the Codazzi equation, we have
\begin{align}\label{eqn:500}
0=&\la(\nabla_XS)Y-(\nabla_YS)X,W_3\ra	
  +2r\left\{\la\phi X,Y\ra-\frac{\la X,V\ra\la Y,B\xi\ra-\la Y,V\ra\la X,B\xi\ra}k\right\} \notag\\
=
& (X\kappa)\la Y,W_3\ra-\frac{r(2-k)t}{k^3}\la X,V\ra\la Y,(\kappa\mathbb I-S)\phi V\ra \notag\\
&-(Y\kappa)\la X,W_3\ra+\frac{r(2-k)t}{k^3}\la Y,V\ra\la X,(\kappa\mathbb I-S)\phi V\ra \notag\\
&-\frac{2(\kappa-t)}{rk}\la X,V\ra\la Y,SB\xi\ra-\frac{2r}k\la X,V\ra\la Y,B\xi\ra 
																																										+\frac{st}k\la X,V\ra\la Y,S\xi\ra\notag\\
&+\frac{2(\kappa-t)}{rk}\la Y,V\ra\la X,SB\xi\ra+\frac{2r}k\la Y,V\ra\la X,B\xi\ra 
																																										   -\frac{st}k\la Y,V\ra\la X,S\xi\ra\notag\\
&+2r\la\phi X,Y\ra+\frac rk\la\kappa(\phi S+S\phi)X-2S\phi SX,Y\ra	
 -\frac{s\kappa}k\la(\phi BS+SB\phi)X,Y\ra.
\end{align}
Next we claim that $U=0$.
For otherwise, we have $\sigma_2=\sigma_4$ or $1+f=2k$. 
It follows that $f=3/5$. Hence  $t=0$ and so $2m-a=0$ by Lemma \ref{lem:B+theta} and  (\ref{eqn:eigen-3}).
By putting $X=W_4$ and $Y\in\mathcal H(1)$ in (\ref{eqn:500}), we obtain $r\la\kappa \phi U-2S\phi U,Y\ra+s\kappa\la\phi U,Y\ra=0$.
Since $S\mathcal H(1)\subset \mathcal H(1)$, we have 
$S\phi U=\mu\phi U$
where 
\[\mu=\frac{r+s}{2r}\kappa=\frac23\kappa.
\]
Since $\mu^2-h\mu=\sigma_1=-8/5$, we have 
\[
\frac49\kappa^2-\frac23h\kappa=-\frac85.
\]
Comparing with (\ref{eqn:eigen-3}), we obtain $\kappa^2=-18/5$; a  contradiction. Hence we conclude that $U=0$ or 
$SW_4=\tau W_4$
and  so
\begin{align}\label{eqn:eigen-4}
\tau^2-h\tau=\sigma_4.
\end{align}
Moreover, we have $S\mathcal H(\varepsilon)\subset\mathcal H(\varepsilon)$, 
$\phi S\phi \mathcal H(\varepsilon)\subset\mathcal H(\varepsilon)$ and $(SB-BS)\mathcal H(\varepsilon)=0$
 for any $\varepsilon\in\{1,-1\}$.

Let $X\in\mathcal H$ and replacing $X$ by $\phi X$ in (\ref{eqn:500}); gives
\[
-2rkX+r\kappa(\phi S\phi X-SX)-2rS\phi S\phi X-s\kappa(\phi BS\phi X-SBX)=0.
\]
By taking the transpose of this equation, we obtain
\begin{align*}
-2rkX+r\kappa(\phi S\phi X-SX)-2r\phi S\phi SX-s\kappa(\phi SB\phi X-BSX)=0.
\end{align*}
It follows that 
\[
2r(\phi S\phi S-S\phi S\phi)X=s\kappa\phi (SB-BS)\phi X-s\kappa(BS-SB)X=0
\]
for any $X\in\mathcal H$. This implies that $S_{|\mathcal H(-1)}$ and $\phi S\phi_{|\mathcal H(-1)}$ are simultaneously diagonalized by 
orthonormal vectors $X_1,\cdots,X_{m-2}$ in $\mathcal H(-1)$, say  
\[
SX_j=\lambda_jX_j, \quad \phi S\phi X_j=-\mu_j X_j, \quad (j\in\{1,\cdots,m-2\}).
\]
It follows that 
\begin{align}\label{eqn:eigen-2}
\lambda_j^2-h\lambda_j=\sigma_2.
\end{align}
Moreover, since each $\phi X_j\in\mathcal H(1)$ and $S\phi X_j=\mu_j\phi X_j$, we also have 
\begin{align}\label{eqn:eigen-1}
\mu_j^2-h\mu_j=\sigma_1.
\end{align}
Letting $X=X_j$ and $Y=\phi X_j$ in (\ref{eqn:500}); gives
\[
2rk+r\kappa(\lambda_j+\mu_j)-2r\lambda_j\mu_j+s\kappa(\lambda_j-\mu_j)=0
\]
which can be rewritten as 
\begin{align}\label{eqn:540}
\left\{(1-k+f)\lambda_j-(1-k-f)\mu_j\right\}\kappa=2f(\lambda_j\mu_j-k).
\end{align}
By a similar calculation, we have
\begin{align*}
(\nabla_XS)W_4
=&(X\tau)W_4-\frac{s(2+k)t}{k^3}\la X,V\ra(\tau\mathbb I-S)\phi V+\frac{2(\tau-t)}{sk}\la X,SB\xi\ra V\\
&+\frac sk(\tau\mathbb I-S)\phi SX- \frac rk(\tau\mathbb I-S)\phi BSX
\end{align*}
and
\begin{align}\label{eqn:510}
0=
& (X\tau)\la Y,W_4\ra-\frac{s(2+k)t}{k^3}\la X,V\ra\la Y,(\tau\mathbb I-S)\phi V\ra \notag\\
&-(Y\tau)\la X,W_4\ra+\frac{s(2+k)t}{k^3}\la Y,V\ra\la X,(\tau\mathbb I-S)\phi V\ra \notag\\
&-\frac{2(\tau-t)}{sk}\la X,V\ra\la Y,SB\xi\ra-\frac{2s}k\la X,V\ra\la Y,B\xi\ra+\frac{rt}k\la X,V\ra\la Y,S\xi\ra \notag\\
&+\frac{2(\tau-t)}{sk}\la Y,V\ra\la X,SB\xi\ra+\frac{2s}k\la Y,V\ra\la X,B\xi\ra-\frac{rt}k\la Y,V\ra\la X,S\xi\ra \notag\\
&-2s\la\phi X,Y\ra+\frac sk\la\tau(\phi S+S\phi)X-2S\phi SX,Y\ra	
 -\frac{r\tau} k\la(\phi BS+SB\phi)X,Y\ra.
\end{align}
Similarly, after putting $X=X_j$ and $Y=\phi X_j$ in (\ref{eqn:510}) gives
\begin{align}\label{eqn:550}
\left\{(1+k+f)\lambda_j-(1+k-f)\mu_j\right\}\tau=2f(\lambda_j\mu_j+k).
\end{align}




In the following calculation, we replace $\lambda_j$ and  $\mu_j$  by $\lambda$ and   $\mu$  respectively for simplicity.
First, after eliminating the variable $h$ in  (\ref{eqn:eigen-3}) and   (\ref{eqn:eigen-4})--(\ref{eqn:eigen-1}),  give 
\begin{align}
\mu\lambda^2=(\mu^2+1+f-\sigma_0)\lambda+(\sigma_0-1+f)\mu		\label{eqn:700} \\
\mu\kappa_\epsilon^2=(\mu^2+1+f-\sigma_0)\kappa_\epsilon+(\sigma_0-2-2\epsilon k)\mu		\label{eqn:710} 
\end{align}
where $\epsilon\in\{1,-1\}$ and we have put $\kappa_{1}=\kappa$  and $\kappa_{-1}=\tau$.
Using  this unify notation,    (\ref{eqn:540}) and (\ref{eqn:550})  can be expressed as  
\begin{align}\label{eqn:720}
\left\{(1-\epsilon k+f)\lambda-(1-\epsilon k-f)\mu\right\}\kappa_\epsilon=2f(\lambda\mu-\epsilon k).
\end{align}
It follows from  (\ref{eqn:710})--(\ref{eqn:720}) that 
\begin{align*}
2\mu f^2(\lambda\mu-\epsilon k)^2
=&(\mu^2+1+f-\sigma_0)f(\lambda\mu-\epsilon k)\left\{(1-\epsilon k+f)\lambda-(1-\epsilon k-f)\mu\right\}	\notag\\
&+\frac\mu 2(\sigma_0-2-2\epsilon k)\left\{(1-\epsilon k+f)\lambda-(1-\epsilon k-f)\mu\right\}^2.
\end{align*}
By applying (\ref{eqn:700}), we can eliminate  the variable $\lambda^2$   in the preceding equation and obtain
\begin{align}\label{eqn:M-2}
k^2\{(\mu^2C_1+C_2)\lambda+(-\mu^2C_1+C_3)\mu\}+\epsilon k \{(\mu^2C_4+C_5)\lambda+(\mu^2C_6+C_7)\mu\}=0
\end{align}
where 
\begin{align*}
C_1=&2f-\sigma_0\\
C_2=&2f(1+f)+(1-f)\sigma_0-\sigma_0^2\\
C_3=&-2f(1+f)-(1-3f)\sigma_0+\sigma_0^2\\
C_4=&-2f(1+f)+\sigma_0\\
C_5=&-2f(1+f)^2+(-1+f+2f^2)\sigma_0+\sigma_0^2\\
C_6=&2f(1-f)-\sigma_0\\
C_7=&2f(1-f^2)+(1-3f)\sigma_0-\sigma_0^2.
\end{align*}
After substituting $\epsilon=\pm1$ in (\ref{eqn:M-2}), we have 
\begin{align}\label{eqn:M-3}
(\mu^2C_1+C_2)\lambda+(-\mu^2C_1+C_3)\mu=(\mu^2C_4+C_5)\lambda+(\mu^2C_6+C_7)\mu=0.
\end{align}
It follows that 
\begin{align}\label{eqn:main-1}
-2C_1\mu^4+D_1\mu^2+D_2=0
\end{align}
where
\begin{align*} 
D_1=&\frac{C_1(C_5+C_7)+C_2C_6-C_3C_4}{2f^2}=-8f(1+f)+8f\sigma_0+\sigma_0^2\\
D_2=&\frac{C_2C_7-C_3C_5}{2f^2}=-4f(1+f)^2+(-2+4f+6f^2)\sigma_0+(3-f)\sigma_0^2-\sigma_0^3.
\end{align*}
On the other hand, by using  (\ref{eqn:700}) and the first equation of (\ref{eqn:M-3}), we obtain
\begin{align}\label{eqn:main-2}
C_1\mu^4+D_3\mu^2+D_4=0
\end{align}
where
\begin{align*}
D_3=&\frac{C_1\{3C_2-C_3+2\sigma_0(\sigma_0-1)\}-2C_3\sigma_0}{4f^2}=4f(1+f)-4f\sigma_0-\sigma_0^2\\
D_4=&\frac{C_2\{C_2-C_3 +2\sigma_0(\sigma_0-1)\}}{4f^2}=2f(1+f)^2+(1-2f-3f^3)\sigma_0-2\sigma_0^2+\sigma_0^3.
\end{align*}
It follows from (\ref{eqn:main-1})--(\ref{eqn:main-2}) 
that $\sigma_0^2(\mu^2+1+f-\sigma_0)=0$.
If  $\sigma_0=0$, then (\ref{eqn:main-2}) reduces to $(\mu^2+1+f)^2=0$, which implies that $f<0$; a contradiction.
Hence we have 
$\mu^2+1+f-\sigma_0=0$.
After substituting this back into (\ref{eqn:main-2}),  gives
$1+f-\sigma_0=0$.
Since $\sigma_0$ is a constant,  $f$ is also a  constant.
By virtue of Lemma~\ref{lem:B+theta}, we have  $t=0$ and so
(\ref{eqn:eigen-3})  implies that $\sigma_0=0$; a contradiction.  
 Consequently, this case does not exist.
\end{proof}

\begin{theorem}\label{thm:Einstein}
Let $M$ be a real hypersurface of the complex quadric $Q^m$, $m\geq3$. 
Then $M$ is pseudo-Einstein, that is, it satisfies (\ref{eqn:Einstein}), 
if and only if one of the following holds
\begin{enumerate}
\item[(a)] $M$ an open part of a tube of radius $r$ around a totally geodesic $Q^{m-1}$ in $Q^m$
where $a=-b=2m$ and $2\cot^2(\sqrt2 r)=m-2$;
\item[(b)] $m=3$, $a=6$, $b=-4$ and $M$ is a $\mathfrak A$-isotropic Hopf hypersurface with principal curvatures
$0$, $\lambda$ and $1/\lambda$.
The corresponding principal curvature space for $0$ is $\mathcal H^\perp$.
Moreover, $\lambda^2\neq1$ on an open dense subset of $M$.
\end{enumerate}
\end{theorem}
\begin{proof}
Suppose $M$ is pseudo-Einstein. According to Theorem~\ref{thm:Eintein0}, we have two cases: $M$ is $\mathfrak A$-principal and $M$ is $\mathfrak A$-isotropic.

\textbf{Case I.} $M$ is $\mathfrak A$-principal.

In this case, we have $f=1$, $g=0$ and $V=0$. Hence, (\ref{eqn:P}) is descended to
\begin{align*}
PX=(2m-a-1)X-(2+b)\eta(X)\xi-BX.
\end{align*}
Since $S\mathcal H(-1)=0$, we obtain $2m-a=0$ and hence
\begin{align}\label{eqn:P:f=1}
PX=-X-(2+b)\eta(X)\xi-BX.
\end{align}

We claim that $M$ is Hopf. Suppose that $M$ is not Hopf. Then we have $b=0$ by Lemma~\ref{lem:schur-2}. It follows that 
$PX=-2$ for any $X$ $\perp\mathcal H(-1)$.
Furthermore, $M$ has three distinct principal curvatures
(for otherwise, $M$ must  be Hopf): 
$0$, $\lambda$ and $\mu$ with multiplicities $m-1$, $m_1$ and $m-m_1$ respectively,
where  $\lambda$ and $\mu$ are solutions for 
\begin{align*}
z^2-hz+2=0.
\end{align*}
Hence, we have $\lambda+\mu=h$ and $\lambda\mu=2$ so that 
\begin{align*}
0
=&m_1\lambda+(m-m_1)\mu-h=\frac{(m_1-1)\lambda^2+2(m-m_1-1)}{\lambda}.
\end{align*}
This contradicts the fact $m\geq 3$. Hence the claim is proved.

By Theorem~\ref{thm:tube} and Theorem~\ref{thm:main}, we conclude that $M$ is an open part of 
a tube of radius $r\in]0,\pi/2\sqrt2[$ around the totally geodesic $Q^{m-1}$ in $Q^m$, 
and 
$M$ has three constant principal curvatures 
\[
\alpha=\sqrt{2}\cot(\sqrt2r), \quad \lambda=-\sqrt{2}\tan(\sqrt2r), \quad \mu=0.
\] 
with multiplicities $1$, $m-1$, $m-1$ respectively.
It follows that 
\[
h=\alpha+(m-1)\lambda.
\]
Moreover, $\alpha$ and $\lambda$ satisfying 
\begin{align*}
\alpha^2-h\alpha+2+b=0, \quad \lambda^2-h\lambda+2=0.
\end{align*}
By using these equations, we obtain $\cot^2(\sqrt2 r)=(m-2)/2$ and $b=-2m$.
This gives Case (a) in the theorem.

\textbf{Case II.} $M$ is $\mathfrak A$-isotropic.

In this case, we have $f=0$ and $SV=S\phi V=0$. Hence, $2m-a=0$ and (\ref{eqn:P}) is descended to
\begin{align*}
PX=-X-(3+b)\eta(X)\xi+\la X,V\ra V+\la X,\phi V\ra\phi V.
\end{align*}
It follows that $P$ has at most three distinct eigenvalues 
\begin{eqnarray*}
	\sigma_0=0, \quad
	\sigma_1=-1, \quad 
	\sigma_2=-4-b 
\end{eqnarray*}
with eigenspaces 
\begin{eqnarray*}
	\mathcal T_0=\mathbb R V\oplus \mathbb \mathbb \phi V, \quad
	\mathcal T_1=\mathcal H, \quad 
	\mathcal T_2=\mathbb R \xi.
\end{eqnarray*}
If $\sigma_2\notin\{\sigma_0,\sigma_1\}$, then $M$ is Hopf as $\dim\mathcal T_2=1$.
On the other hand, 
If $\sigma_2\in\{\sigma_0,\sigma_1\}$, then $b\neq0$ and so $M$ is also Hopf by Lemma~\ref{lem:schur-2}.
Hence, we conclude that $M$ is Hopf in this case.
We take an orthonormal basis $\{X_1,\cdots,X_{m-2},\phi X_1,\cdots,\phi X_{m-2}\}$ in $\mathcal H$ such that 
\[
SX_j=\lambda_j X_j, \quad  S\phi X_j=\mu_j\phi X_j \quad (j\in\{1,\cdots,m-2\}).
\]
By (\ref{eqn:210}), we have
\begin{align}\label{eqn:210b}
2\lambda_j\mu_j-\alpha(\lambda_j+\mu_j)-2=0.
\end{align}
Moreover, each $\lambda_j$, $\mu_j$ must be solutions of
\begin{align}\label{eqn:f=0}
z^2-hz+1=0.
\end{align}
We consider 
\[
\mathcal E_j=\{[z]:\lambda_j=\mu_j\}; \quad \mathcal E=\bigcap_{j=1}^{m-2}\mathcal E_j.
\]
If $\Int \mathcal E\neq\emptyset$, then we have $\phi S-S\phi=0$ on $\Int \mathcal E$
and by a result in \cite{berndt-suh2}, there are four principal curvatures: $\alpha=2\cot 2r$, $\lambda_1=\cot r$, $\lambda_2=-\tan r$, $\beta=0$.
The corresponding principal curvature spaces are 
\[\mathcal T_{\alpha}=\mathbb R\xi, \quad \mathcal T_{\lambda_1}, \quad \mathcal T_{\lambda_2}, \quad 
\mathcal T_{\beta}=\mathbb R V\oplus\mathbb R\phi V
\] 
where $\phi \mathcal T_{\lambda_1}=\mathcal T_{\lambda_1}$, $\phi \mathcal T_{\lambda_2}=\mathcal T_{\lambda_2}$ and 
$\mathcal H=\mathcal T_{\lambda_1}\oplus\mathcal T_{\lambda_2}$.
Since $\lambda_1$, $\lambda_2$ are solutions of (\ref{eqn:f=0}),
we have $\lambda_1\lambda_2=1$. This is a contradiction and so $\Int \mathcal E=\emptyset$.

Without loss of generality, we assume that $\mathcal E^c_1\neq\emptyset$.
It follows that $\lambda_1$, $\mu_1$ are distinct solutions of (\ref{eqn:f=0}) on $\mathcal E^c_1$. 
Hence, $\lambda_1+\mu_1=h$ and $\lambda_1\mu_1=1$.
Substituting these into (\ref{eqn:210b}) gives $h\alpha=0$. Hence, $\alpha=0$ in view of (\ref{eqn:f=0})
and so $4+b=-\sigma_2=0$.
Suppose $m\geq4$. 
If there exists $j\in\{2,\cdots,m-2\}$ such that $\mathcal E^c_1\cap\Int \mathcal E_j\neq\emptyset$, 
since there are only three principal curvatures in this case, we have $\lambda_j\in\{\lambda_1,\mu_1\}$, say $\lambda_j=\lambda_1$.
It follows from (\ref{eqn:210b}) that $\lambda_j^2=\lambda_1\mu_1$.
This contradicts the assumption $\lambda_1\neq\mu_1$.
Hence we conclude that $\mathcal E^c_1\supset\cap_{j=1}^{m-2}\mathcal E^c_j\neq\emptyset$. It follows that 
\[
0=(m-2)(\lambda_1+\mu_1)-h=(m-3)h.
\]
This is a contradiction. Hence $m=3$ and $\mathcal E^c=\mathcal E^c_1$, which is open and dense in $M$.
This gives Statement (b).
The converse is trivial.
\end{proof}

\begin{remark}
\begin{enumerate}
\item[(a)]
Pseudo-Einstein Hopf hypersurfaces in $Q^m$ were studied in \cite{suh}. However, the classification was incomplete as we have shown 
the real hypersurfaces listed in \cite[Main Theorem 2(ii)]{suh} is not pseudo-Einstein.
\item[(b)] The author does not know any example of real hypersurfaces stated in Theorem~\ref{thm:Einstein}(b).
However,  even if it exists, this example is local in the sense that it is not extendible to a complete real hypersurface on the basis of 
Theorem~\ref{thm:Einstein-2} below.
\end{enumerate}
\end{remark}

\begin{theorem}\label{thm:Einstein-2}
Let $M$ be a complete real hypersurface of the complex quadric $Q^m$, $m\geq3$. 
Then $M$ is pseudo-Einstein, that is, it satisfies (\ref{eqn:Einstein}), 
 if and only if it is congruent to 
 a tube of radius $r$ around a totally geodesic $Q^{m-1}$ in $Q^m$
where $a=-b=2m$ and $2\cot^2(\sqrt2 r)=m-2$.
\end{theorem}
\begin{proof}
Suppose $M$ is complete pseudo-Einstein and satisfies the properties in Theorem~\ref{thm:Einstein}(b).
Then we have $f=g=0$ and $k=1$.  
Taking a unit vector field $X$ tangent to $\mathcal H$ with $SX=\lambda X$ and 
$S\phi X=(1/\lambda)\phi X$.  
Furthermore,  taking the reciprocal if necessary,  we have $\lambda^2\geq 1$.
Note that $\la S,S\ra=\lambda^2+(1/\lambda)^2\geq 2$ with equality holds if and only if $\lambda^2=1$.
Since $M$ is compact, $\la S,S\ra$ is bounded.
Suppose the maximum for $\la S,S\ra$ attained at a  point $[z]\in M$. Then $\lambda^2>1$  and 
so $\lambda$ is differentiable at  $[z]$
by Theorem~\ref{thm:Einstein}(b).

By the Codazzi equation, we have 
\[
\la (\nabla_XS)V-(\nabla_VS)X,X\ra=-\la BX,X\ra.
\]
At the points on which $\lambda$ is differentiable, by applying Lemma~\ref{lem:B+theta} and Lemma~\ref{lem:pri+iso}
to the preceding equation, we have 
\begin{align}\label{eqn:Vl}
V\lambda=(1+\lambda^2)\la BX,X\ra.
\end{align}
Similarly, 
we have  
\begin{align}\label{eqn:Vl-2}
\phi V\lambda
=&\la(\nabla_{\phi V}S)X,X\ra =\la (\nabla_XS)\phi V,X\ra-\la BX,\phi X\ra \nonumber\\
=&-(1+\lambda^2)\la BX,\phi X\ra.
\end{align}
Since $[z]$ is a critical point, 
it follows from (\ref{eqn:Vl})--(\ref{eqn:Vl-2}) that $\la BX,X\ra=\la BX,\phi X\ra=0$ at the point $[z]$.
Since $\dim\mathcal H=2$ and $B\mathcal H\subset\mathcal H$, 
we get $BX=\la BX,X\ra X+\la BX,\phi X\ra\phi X=0$. This is a contradiction and the proof is completed.
\end{proof}

\bibliographystyle{amsplain}

\end{document}